\documentclass[a4paper,10pt]{amsart}
\usepackage[english]{babel}
\usepackage[active]{srcltx}
\usepackage{amsmath,amssymb,amsthm}
\usepackage{multicol}

\usepackage[pdftex]{color}
\usepackage[bookmarks=true,hyperindex,pdftex,colorlinks, citecolor=blue,urlcolor=blue]{hyperref}

\hypersetup{pdftitle={Two refinements of the Bishop-Phelps-Bollob\'{a}s modulus}, pdfauthor={Chica, Kadets, Martin, Meri, Soloviova}}

\usepackage{graphicx}

\usepackage{pgf,tikz}
\usetikzlibrary{arrows}

\theoremstyle{plain}
\newtheorem{theorem}{Theorem}[section]
\newtheorem{corollary}[theorem]{Corollary}
\newtheorem{prop}[theorem]{Proposition}
\newtheorem{lemma}[theorem]{Lemma}

\theoremstyle{definition}

\newtheorem{remark}[theorem]{Remark}
\newtheorem{example}[theorem]{Example}
\newtheorem{examples}[theorem]{Examples}
\newtheorem{definition}[theorem]{Definition}

 \DeclareMathOperator{\re}{Re\,}

 \DeclareMathOperator{\dist}{dist\,}

\newcommand{\R}{\mathbb{R}}

\newcommand{\norm}[1]{\ensuremath{\lVert#1\rVert}}

\newcommand{\nor}[1]{\ensuremath{\left\|#1\right\|}}
\newcommand{\abs}[1]{\ensuremath{\lvert#1\rvert}}
\newcommand{\Abs}[1]{\ensuremath{\left|#1\right|}}

\newcommand{\eps}{\varepsilon}

\renewcommand{\leq}{\leqslant}
\renewcommand{\geq}{\geqslant}

\renewcommand{\ge}{\geqslant}

\newcommand{\fr}[2]{\frac{#1}{#2}}

\newcommand{\lf}{\left\{ }
\newcommand{\rt}{\right\} }

\begin{document}
\title{Two refinements of the Bishop-Phelps-Bollob\'{a}s modulus}

\author[Chica]{Mario Chica}
\author[Kadets]{Vladimir Kadets}
\author[Mart\'{\i}n]{Miguel Mart\'{\i}n}
\author[Mer\'{\i}]{Javier Mer\'{\i}}
\author[Soloviova]{Mariia Soloviova}

\address[Chica]{Departamento de An\'{a}lisis Matem\'{a}tico \\ Facultad de
 Ciencias \\ Universidad de Granada \\ 18071 Granada, Spain}
\email{mcrivas@ugr.es}

\address[Kadets]{Department of Mechanics and Mathematics,
Kharkiv National University, pl.~Svobody~4,
\newline
61077~Kharkiv, Ukraine}
\email{vova1kadets@yahoo.com}

\address[Mart\'{\i}n]{Departamento de An\'{a}lisis Matem\'{a}tico \\ Facultad de
 Ciencias \\ Universidad de Granada \\ 18071 Granada, Spain
\newline
\href{http://orcid.org/0000-0003-4502-798X}{ORCID: \texttt{0000-0003-4502-798X} }
 }
\email{mmartins@ugr.es}

\address[Mer\'{\i}]{Departamento de An\'{a}lisis Matem\'{a}tico \\ Facultad de
 Ciencias \\ Universidad de Granada \\ 18071 Granada, Spain}
\email{jmeri@ugr.es}

\address[Soloviova]{Department of Mechanics and Mathematics,
Kharkiv National University, pl.~Svobody~4,
\newline
61077~Kharkiv, Ukraine}
\email{5\_11\_16@mail.ru}

\thanks{The work of the first, third and fourth authors has been partially supported by Spanish MICINN and FEDER project no.\ MTM2012-31755 and by Junta de Andaluc\'{\i}a and FEDER grants FQM-185 and P09-FQM-4911.}

 \subjclass[2010]{Primary: 46B04}
 \keywords{Banach space; Bishop-Phelps theorem; approximation; uniformly non-square spaces}

\begin{abstract}
Extending the celebrated result by Bishop and Phelps that the set of norm attaining functionals is always dense in the topological dual of a Banach space, Bollob\'{a}s proved the nowadays known as the Bishop-Phelps-Bollob\'{a}s theorem, which allows to approximate at the same time a functional and a vector in which it almost attains the norm. Very recently, two Bishop-Phelps-Bollob\'{a}s moduli of a Banach space have been introduced [\emph{J.\ Math.\ Anal.\ Appl.}~412 (2014), 697--719] to measure, for a given Banach space, what is the best possible Bishop-Phelps-Bollob\'{a}s theorem in this space. In this paper we present two refinements of the results of that paper. On the one hand, we get a sharp general estimation of the Bishop-Phelps-Bollob\'{a}s modulus as a function of the norms of the point and the functional, and we also calculate it in some examples, including Hilbert spaces. On the other hand, we relate the modulus of uniform non-squareness with the Bishop-Phelps-Bollob\'{a}s modulus obtaining, in particular, a simpler and quantitative proof of the fact that a uniformly non-square Banach space cannot have the maximum value of the Bishop-Phelps-Bollob\'{a}s modulus.
\end{abstract}

\date{October 27th, 2014}

\maketitle

\section{Introduction}
It is a celebrated result of the geometry of Banach spaces that the set of norm-attaining functionals is always dense in the topological dual of a Banach space (i.e.\ the classical Bishop-Phelps theorem of 1961 \cite{BishopPhelps61}). A refinement of this theorem, nowadays known as the Bishop-Phelps-Bollob\'{a}s theorem \cite{Bollobas}, was proved by B.~Bollob\'{a}s and allows to approximate at the same time a functional and a vector in which it almost attains the norm. Very recently, two moduli have been introduced \cite{C-K-M-M-R} which measure, for a given Banach space, what is the best possible Bishop-Phelps-Bollob\'{a}s theorem in this space. We need some notation. Given a (real or complex) Banach space $X$, $X^*$ denotes the (topological) dual of
$X$. We write $B_X$ and $S_X$ to denote respectively the closed unit ball and the unit
sphere of the space. We consider the set in $B_X\times B_{X^*}$ given by
$$
\Pi(X):=\bigl\{(x,x^*)\in X\times X^*\,:\,
\|x\|=\|x^*\|=x^*(x)=1\bigr\}.
$$

\begin{definition}[Bishop-Phelps-Bollob\'{a}s moduli, \cite{C-K-M-M-R}] \label{def1.2bpb-mod}$ $\newline Let $X$ be a Banach space. The \emph{Bishop-Phelps-Bollob\'{a}s modulus} of $X$ is the function $\Phi_X:(0,2)\longrightarrow \R^+$ such that given $\delta\in (0,2)$, $\Phi_X(\delta)$ is the infimum of those $\eps>0$ satisfying that for every $(x,x^*)\in B_X\times B_{X^*}$ with $\re x^*(x)>1-\delta$, there is $(y,y^*)\in \Pi(X)$ with
$\|x-y\|<\eps$ and $\|x^*-y^*\|<\eps$.

The \emph{spherical Bishop-Phelps-Bollob\'{a}s modulus} of $X$ is the function $\Phi^S_X:(0,2)\longrightarrow \R^+$ such that given $\delta\in (0,2)$, $\Phi^S_X(\delta)$ is the infimum of those $\eps>0$ satisfying that for every $(x,x^*)\in S_X\times S_{X^*}$ with $\re x^*(x)>1-\delta$, there is $(y,y^*)\in \Pi(X)$ with
$\|x-y\|<\eps$ and $\|x^*-y^*\|<\eps$.
\end{definition}

If we equip $X\times X^*$ with the metric given by
$$
d_\infty\big((x,x^*),(y,y^*)\big)=\max\{\|x-y\|,\|x^*-y^*\|\}, \,\,  (x,x^*), (y,y^*)\in X\times X^*,
$$
it is clear that $\Phi_X(\delta)$ and $\Phi_X^S(\delta)$ are the Hausdorff distances to $\Pi(X)$ of, respectively, the sets
\begin{equation*}
\bigl\{(x,x^*)\in B_X\times B_{X^*}\,:\, \re
x^*(x)>1-\delta \bigr\} \ \text{ and } \ \bigl\{(x,x^*)\in S_X\times S_{X^*}\,:\, \re
x^*(x)>1-\delta \bigr\}.
\end{equation*}

Observe that, clearly, $\Phi_X^S(\cdot)\leq \Phi_X(\cdot)$ and there are examples in which the inequality is strict (see \cite[\S4]{C-K-M-M-R}. More interesting properties of both moduli can be found in the cited paper \cite{C-K-M-M-R}, where we refer for background.

One of the main results of \cite{C-K-M-M-R} states that there is a common upper bound for $\Phi_X(\cdot)$ (and so for $\Phi_X^S(\cdot)$) for all Banach spaces which is actually sharp. Namely, it is shown that for every Banach space $X$ and every $\delta\in (0,2)$ one has
$\Phi_X(\delta)\leq \sqrt{2\delta}$. In other words, this leads to the following improved version of the Bishop-Phelps-Bollob\'{a}s theorem.

\begin{prop}[Sharp version of the Bishop-Phelps-Bollob\'{a}s theorem, \mbox{\cite[Corollary~2.4]{C-K-M-M-R}}] \label{Prop1.2}
Let $X$ be a Banach space and $0<\eps<2$. Suppose that $x\in B_X$ and $x^*\in B_{X^*}$ satisfy $\re x^*(x) > 1- \eps^2/2$. Then, there exists $(y,y^*)\in \Pi(X)$ such that $\|x-y\|<\eps$ and $\|x^*-y^*\|<\eps$.
\end{prop}

This version is best possible \cite[Example~2.5]{C-K-M-M-R} by just considering $X=\ell_\infty^{(2)}$, the two-dimensional real $\ell_\infty$ space.

It is observed in \cite[Remark~2.3]{C-K-M-M-R} that a stronger version can be deduced when considering non-unital functionals:
\begin{quote}
For every $0<\theta<1$ and every $0<\delta<2$, there is $\rho = \rho(\delta, \theta) > 0$ such that for every Banach space $X$, if $x^*\in B_{X^*}$ with $\|x^*\| \leq \theta$, $x\in B_X$ satisfy that $\re x^*(x) > 1-\delta$, then
$$
d_\infty\big((x,x^*),\Pi(X)\big)<\sqrt{2\delta} - \rho
$$
\end{quote}
where, as usual,
$$
d_\infty\big((x,x^*),\Pi(X)\big)=\inf\big\{\max\{\|x-y\|,\|x^*-y^*\|\} \ : \ (y,y^*)\in \Pi(X)\big\}.
$$

The first goal of the present paper is to deal with the problem of calculating the best possible upper bound for $d_\infty\big((x,x^*),\Pi(X)\big)$ in an arbitrary Banach space $X$ as a function of $\|x\|$ and $\|x^*\|$. More precisely, given a Banach space $X$ and fixed $\delta \in(0,2)$ and $\mu,\theta\in [0,1]$ satisfying $\mu\theta\geq 1-\delta$, we consider
$$
\Phi_X(\mu,\theta,\delta):=\sup\left\{d_\infty\big((x,x^*),\Pi(X)\big) \ : \ x\in X, x^*\in X^*, \|x\|=\mu, \|x^*\|=\theta, \re x^*(x)\geq 1-\delta\right\}.
$$
In section~\ref{sec:non-unital}, we will provide an estimation for $\Phi_X(\mu,\theta,\delta)$ valid for every Banach space $X$ and present examples showing that the estimation is sharp. We further calculate $\Phi_X(\mu,\theta,\delta)$ in some particular cases, including Hilbert spaces.

In the second part of this manuscript, which is contained in section~\ref{sec:non-square}, we deal with another refinement of Proposition \ref{Prop1.2}. Namely, in \cite[Theorem~5.9]{C-K-M-M-R} it is proved that for a uniformly non-square space $X$ and $\delta\in(0,\frac12)$ one has
\begin{equation*}
\Phi_X^S(\delta) < \sqrt{2\delta}.
\end{equation*}
The proof of this fact is involved and it is impossible to extract from it any estimate for $\Phi_X^S(\delta)$. Our goal in section~\ref{sec:non-square} is to give a simpler proof that provides a quantification of the inequality above in terms of a parameter that measures the uniformly non-squareness of the Banach space $X$.

\section{The modulus for non-unital points and functionals}\label{sec:non-unital}

For clearness of the arguments in this section, let us use the following notation. For $\delta\in (0,2)$ and $\mu,\theta\in [0,1]$ with $\mu\theta> 1-\delta$, we define the function
$$
\Psi(\mu,\theta,\delta):=\frac{2-\mu-\theta+\sqrt{(\mu -\theta)^2+8\left( \mu \theta -1 + \delta \right)}}{2}.
$$
The main result of this section is the following improvement of \cite[Theorem 2.1]{C-K-M-M-R} which quantifies \cite[Remark~2.3]{C-K-M-M-R}.

\begin{theorem}\label{thm:cota-general}
Let $X$ be a Banach space, $\delta\in(0,2)$, and $\mu,\theta\in [0,1]$ satisfying $\mu\theta> 1-\delta$. Then,
$$
\Phi_X(\mu,\theta,\delta)\leq \min\left\{\Psi(\mu,\theta,\delta), 1+\mu,1+\theta\right\}.
$$
\end{theorem}

Let us provide some preliminary results needed in the proof of this theorem. The first one gives an easy inequality and also covers  what happens in the trivial case in which $\mu\theta=1-\delta$.

\begin{remark}\label{rem:lower-bound-Phi} {\slshape Let $X$ be a Banach space, $\delta\in(0,2)$, and $\mu,\theta\in [0,1]$ satisfying $\mu\theta\geq 1-\delta$. Then, the inequality $\Phi_X(\mu,\theta,\delta)\geq 1-\min\{\mu,\theta\}$ holds. Moreover, if $\mu\theta= 1-\delta$, in fact one has $\Phi_X(\mu,\theta,\delta)=1-\min\{\mu,\theta\}$.\ } \newline Indeed, fix a pair $(x_0, x^*_0)\in \Pi(X)$ and  write $x=\mu x_0$ and $x^*= \theta x_0^*$. Then, it is clear that $x^*(x)\geq 1-\delta$ and
\begin{align*}
\Phi_X(\mu,\theta,\delta)&\geq d_\infty\big((x,x^*),\Pi(X)\big)=\underset{(y,y^*)\in \Pi(X)}{\inf} \max \{\|x-y\|,\|x^*-y^*\|\}\\&\geq \underset{(y,y^*)\in \Pi(X)}{\inf} \max \{1-\mu,1-\theta\}=1-\min\{\mu,\theta\}.
\end{align*}
To prove the moreover part, given any pair $(x,x^*)\in X\times X^*$ satisfying $\|x\|=\mu$, $\|x^*\|=\theta$, and $\re x^*(x)\geq 1-\delta$ we first observe that, in this case, $\re x^*(x)=1-\delta$. Now, if $\mu\theta>0$ we take $y=\frac{x}{\mu}$ and $y^*=\frac{x^*}{\theta}$ which satisfy $\re y^*(y)=1$ and
$$
d_\infty\big((x,x^*),\Pi(X)\big)\leq \max\{\|x-y\|,\|x^*-y^*\|\}=1-\min\{\mu,\theta\}.
$$
Taking supremum in $(x,x^*)$, we get $\Phi_X(\mu,\theta,\delta)\leq 1-\min\{\mu,\theta\}$. If $\mu\theta=0$, an analogous argument with the obvious simplifications gives the desired inequality.
\end{remark}

Next, we provide some elementary observations on the function $\Psi$ whose proof is straightforward.

\begin{lemma}
For $\delta\in (0,2)$ and $\mu,\theta\in [0,1]$ with $\mu\theta> 1-\delta$ we have
\begin{enumerate}
\item[a)] $\Psi(\mu,\theta,\cdot)$ is non-decreasing.
\item[b)] $\Psi(\mu,\theta,\delta)=\Psi(\theta,\mu,\delta)$.
\item[c)] $\Psi(\mu,\theta,1+\mu^2)=1+\mu$.
\item[d)] If $\delta \leq 1$, then $\Psi(\mu,\theta,\delta)\leq 1 + \mu$ and $\Psi(\mu,\theta,\delta)\leq 1 + \theta$.
\end{enumerate}
\end{lemma}

Finally, we will need the following result from \cite{Phelps} which we state for the sake of clearness.

\begin{lemma}[\mbox{Particular case of \cite[Corollary~2.2]{Phelps}}] \label{lemma:corolario2.2-phelps}
Suppose $C$ is a closed convex subset of the Banach space $X$, that $z^*\in S_{X^*}$ and that $\eta > 0$ and $z\in C$ are such that $$\sup z^*(C)\leq z^*(z) + \eta.$$
Then, for any $k\in(0,1)$ there exist $\tilde{y}^*\in X^*$ and $\tilde{y}\in C$ such that
\begin{equation*}
 \sup \tilde{y}^*(C)=\tilde{y}^*(\tilde{y}),\qquad \norm{z-\tilde{y}} < \frac{\eta}{k}, \qquad \norm{z^*-\tilde{y}^*} < k.
\end{equation*}
\end{lemma}

\begin{proof}[Proof of Theorem~\ref{thm:cota-general}] Fixed $(x,x^*)\in X\times X^*$ satisfying $\|x\|=\mu$, $\|x^*\|=\theta$, and $\re x^*(x)\geq1-\delta$, we take $y_1\in S_X$ satisfying $\|x-y_1\|\leq 1$, $y_1^*\in S_{X^*}$ so that $y_1^*(y_1)=1$ and observe that
$$
\max\{\|x-y_1\|,\|x^*-y_1^*\|\}\leq 1+\theta.
$$
We can produce a dual argument by means of the Bishop-Phelps theorem: find a norm attaining functional $y_2^*\in S_{X^*}$ with $\|x^*-y_2^*\|\leq 1$ and a point $y_2\in S_X$ satisfying $y_2^*(y_2)=1$. Then, we have that
$$
\max\{\|x-y_2\|,\|x^*-y_2^*\|\}\leq 1+\mu
$$
and, therefore,
$$
d_\infty\big((x,x^*),\Pi(X)\big)\leq \min\left\{ 1+\mu,1+\theta\right\}.
$$
Now, since $\Psi(\mu,\theta,1+\mu^2)=1+\mu$, $\Psi(\mu,\theta,1+\theta^2)=1+\theta$, and $\Psi(\mu,\theta,\cdot)$ is a non-decreasing function, the proof will be finished if we show that
$$
d_\infty\big((x,x^*),\Pi(X)\big)\leq \Psi(\mu,\theta,\delta)
$$
for $\delta <\min\{1+\mu^2,1+\theta^2\}$. In this case $\mu\theta>1-\delta>-\theta^2$ which implies that $\theta>0$.
Thus we can define
$$
\eta =\frac{\mu \theta-1 + \delta}{ \theta}>0, \qquad z= x, \qquad  \text{and} \qquad
z^*= \frac{x^*}{\theta}
$$ which satisfy $\re z^*(z)> \mu-\eta$. Besides, consider
$$
k = \frac{\theta-\mu+\sqrt{(\mu -\theta)^2+8 (\mu\theta-1+\delta)}}{4\theta}\,
$$
It is clear that $k>0$ and, using the fact that $\delta<1+\theta^2$, it is not difficult to verify that $k<1$:
\begin{align*}
k < \frac{\theta-\mu+\sqrt{(\mu -\theta)^2+8 (\mu\theta+\theta^2)}}{4\theta}=\frac{\theta-\mu+\sqrt{(\mu+3\theta)^2}}{4\theta}=1.
\end{align*}
Therefore, we may apply Lemma~\ref{lemma:corolario2.2-phelps} for $C= \mu B_X$, $z^* \in S_{X^*}$, $z\in B_X$, $\eta >0$, and $0<k<1$ to obtain $\widetilde{y}^*\in X^*$ and $\widetilde{y}\in C$ satisfying
$$
\widetilde{y}^*(\widetilde{y})=\sup \widetilde{y}^*(C)=\mu\norm{\widetilde{y}^*},\qquad \norm{z-\widetilde{y}}<\frac{\eta}{k}, \qquad \text{and} \qquad \norm{z^*-\widetilde{y}^*}=\nor{\frac{x^*}{\theta}-\widetilde{y}^*}<k.
$$
As $k<1$ we get $\widetilde{y}^*\neq 0$ and we can write $y^*=\frac{\widetilde{y}^*}{\norm{\widetilde{y}^*}}$, $y=\frac{\widetilde{y}}{\mu}$, to obtain that $(y,y^*)\in \Pi(X)$. This way, we have that
\begin{align*}
\norm{x-y} &= \nor{z-\frac{\widetilde{y}}{\mu}} \leq \norm{z-\widetilde{y}} + \nor{\widetilde{y}-\frac{\widetilde{y}}{\mu}}< \frac{\eta}{k} +1 -\mu.
\end{align*}
On the other hand we can estimate $\norm{x^*-y^*}$ as follows:
\begin{align*}
\norm{x^*-y^*}&= \nor{x^*-\frac{\widetilde{y}^*}{\norm{\widetilde{y}^*}}}\leq \big\|x^*-\theta\widetilde{y}^*\big\| + \nor{\theta\widetilde{y}^* - \frac{\widetilde{y}^*}{\norm{\widetilde{y}^*}}}\\& \leq \theta\nor{\frac{x^*}{\theta}-\widetilde{y}^*}+\big|\theta\norm{\widetilde{y}^*}-1\big| \leq  \theta\nor{\frac{x^*}{\theta}-\widetilde{y}^*} + \big|\theta\norm{\widetilde{y}^*}-\theta\big|  + \big|1-\theta\big|\\ &\leq \theta\left(\nor{\frac{x^*}{\theta}-\widetilde{y}^*} +\big|\norm{\widetilde{y}^*}-1\big|\right) +1 -\theta\leq 2\theta\nor{\frac{x^*}{\theta}-\widetilde{y}^*}+1-\theta\\ &< 2k\theta+1-\theta.
\end{align*}
Finally, is is routine to check that $\frac{\eta}{k}+1-\mu=2k\theta+1-\theta=\Psi(\mu,\theta,\delta)$. Therefore,
$$
d_\infty\big((x,x^*),\Pi(X)\big)\leq\max\{\norm{x-y},\norm{x^*-y^*}\} < \Psi(\mu,\theta,\delta)
$$
which finishes the proof.
\end{proof}

Our next aim is to present an example for which the estimation given in Theorem~\ref{thm:cota-general} is sharp.
Taking into account \cite[Example~2.5]{C-K-M-M-R} the reasonable candidate is the real space $\ell_\infty^{(2)}$.

\begin{example}\label{example:sharp-mod}
Let $X$ be the real space $\ell_\infty^{(2)}$, $\delta \in (0,2)$, and $\mu$, $\theta \in [0,1]$ satisfying $\mu\theta> 1-\delta$. Then, there exists a pair $(x,x^*)\in X\times X^*$ with $\norm{x}=\mu$, $\norm{x^*}=\theta$,  $x^*(x)\geq 1-\delta$, and such that
$$
d_\infty\big((x,x^*),\Pi(X)\big)=\min\left\{\Psi(\mu,\theta,\delta), 1+\mu,1+\theta\right\}.
$$
Therefore, $\Phi_X(\mu,\theta,\delta)=\min\left\{\Psi(\mu,\theta,\delta), 1+\mu,1+\theta\right\}$ for all possible values of $\delta$, $\mu$, $\theta$.
\end{example}

\begin{proof}
We divide the proof into three cases depending on the expression in which the minimum is attained.

\noindent\emph{Case 1:} $\min\left\{\Psi(\mu,\theta,\delta), 1+\mu,1+\theta\right\}=\Psi(\mu,\theta,\delta)$.

Since $\Psi(\mu,\theta,\cdot)$ is a non-decreasing function and $\Psi(\mu,\theta,1+\theta^2)=1+\theta\leq \Psi(\mu,\theta,\delta)$ we have that $\delta\leq 1+\theta^2$. Thus, we can write $\mu\theta> 1-\delta\geq-\theta^2$ which implies $\theta>0$, so we can define
$$
k = \frac{\theta-\mu+\sqrt{(\mu -\theta)^2+8( \mu\theta-1+\delta)}}{4\theta}\,,\qquad
x=(\mu,1-\Psi(\mu,\theta,\delta))\qquad\mbox{and}\qquad x^*=(\theta(1-k),\theta k).
$$
As we observed in the proof of Theorem~\ref{thm:cota-general}, $k\in (0,1)$ and so $\|x^*\|=\theta$.
Besides, we can estimate as follows
$$
\Psi(\mu,\theta,\delta))\geq \frac{2-\mu-\theta+\sqrt{(\mu -\theta)^2}}{2}\geq 1-\mu.
$$
This, together with the fact that $\Psi(\mu,\theta,\delta))\leq 1+\mu$, allows us to get the equality $\|x\|=\mu$. Moreover, we have that
\begin{align*}
x^*(x)&=\mu\theta(1-k)+ (1-\Psi(\mu,\theta,\delta))\theta k=\mu\theta+ (1-\mu-\Psi(\mu,\theta,\delta))\theta k \\
&= \mu\theta+ \frac{\theta-\mu-\sqrt{(\mu -\theta)^2+8\left( \mu \theta -1 + \delta \right)}}{2}\quad\frac{\theta-\mu+\sqrt{(\mu -\theta)^2+8\left( \mu \theta -1 + \delta \right)}}{4}\\
&= \mu\theta-(\mu\theta-1+\delta)=1-\delta.
\end{align*}
In view of Theorem~\ref{thm:cota-general}, to finish the proof in this case we only need to show that
$$
d_\infty\big((x,x^*),\Pi(X)\big)\geq \Psi(\mu,\theta,\delta).
$$
Fixed $(y,y^*)\in \Pi(X)$ there are $a,b,c,d\in \R$ such that $y=(a,b)$, $y^*=(c,d)$ and
$$
\max\{|a|,|b|\}=1 \qquad |c|+|d|=1\qquad \text{and}\qquad ac+bd=1.
$$
We distinguish two cases depending on the values of $d$. Suppose first that $d\leq 0$ and recall that $k\geq 0$ to write
$$
\|x^*-y^*\|=\big|c-\theta(1-k)\big|+\big|d-\theta k\big|\geq |c|-\theta(1-k)+ |d|+\theta k=2\theta k+1-\theta=\Psi(\mu,\theta,\delta).
$$
If otherwise $d>0$, then the inequality
$$
|c|+|d|=1=ac+bd\leq |c|+bd
$$
yields that $b=1$ and we can write
$$
\|x-y\|=\max\{|\mu-a|,|\Psi(\mu,\theta,\delta)|\}\geq \Psi(\mu,\theta,\delta),
$$
which finishes the proof for Case 1.

\noindent \emph{Case 2:} $\min\left\{\Psi(\mu,\theta,\delta), 1+\mu,1+\theta\right\}=1+\theta$.

In this case we have that $\delta\geq 1+\theta^2$ and $\mu\geq \theta$. So defining
$$
x=(\mu,-\theta) \qquad \text{and} \qquad x^*=(0,\theta),
$$
it is clear that $\|x\|=\mu$, $\|x^*\|=\theta$, and $x^*(x)=-\theta^2\geq 1-\delta$. To verify that $d_\infty\big((x,x^*),\Pi(X)\big)\geq 1+\|x^*\|$ one can proceed analogously to the previous case.

\noindent \emph{Case 3:} $\min\left\{\Psi(\mu,\theta,\delta), 1+\mu,1+\theta\right\}=1+\mu$.

In this case one has that $\delta\geq 1+\mu^2$ and $\theta\geq \mu$. So
$$
x=(\mu,-\mu)\qquad \text{\and} \qquad x^*=\left(\frac{\theta-\mu}{2},\frac{\theta+\mu}{2}\right)
$$
fulfill the desired conditions.
\end{proof}

\subsection{Further examples for which the estimate of $\Phi_X$ is sharp}

In the following we give more examples for which the estimation in Theorem~\ref{thm:cota-general} is sharp. We start with spaces admitting an $L$-decomposition, in particular $L_1(\mu)$-spaces.

\begin{prop}\label{prop:ell1-sum}
Let $X$ be a Banach space. Suppose that there are two (non-trivial) subspaces $Y$ and $Z$ such that $X=Y\oplus_1 Z$. Let $\delta\in (0,1)$ and $\mu,\theta\in [0,1]$ with $1-\delta< \mu\theta \leq 2(1-\delta)$. Then, there exists a pair $(x_0,x_0^*)\in X\times X^*$ with $\norm{x_0}=\mu$, $\norm{x_0^*}=\theta$ and $\re x_0^*(x_0)\geq 1-\delta$ satisfying
$$
d_\infty\big((x_0,x_0^*),\Pi(X)\big)=\Psi(\mu,\theta,\delta).
$$
Therefore, $\Phi_X(\mu,\theta,\delta)=\Psi(\mu,\theta,\delta)$ for the cited values of $\delta$, $\mu$, $\theta$.
\end{prop}

\begin{proof} Since $0<1-\delta<\mu\theta$ we get that $\mu>0$, so we can take $k = \frac{\mu-\theta+\sqrt{(\mu -\theta)^2+8( \mu\theta-1+\delta)}}{4\mu}$ which satisfies $0\leq k\leq 1$ because $\delta<1$. Next, we fix pairs $(y_0,y_0^*)\in \Pi(Y)$ and $(z_0,z_0^*)\in \Pi(Z)$, and we define
$$
x_0=\left(\mu ky_0,\mu(1-k)z_0\right)\qquad \text{and}\qquad x_0^*=\left((1-\Psi(\mu,\theta,\delta)) y_0^*,\theta z_0^*\right).
$$
The facts $|1-\Psi(\mu,\theta,\delta)|\leq\theta$ and $0\leq k\leq1$ imply that $\norm{x_0}=\mu$, $\norm{x_0^*}=\theta$.
Moreover, we can write
\begin{align*}
\re x_0^*(x_0)&=\mu\theta+ (1-\theta-\Psi(\mu,\theta,\delta))\mu k\\
&=\mu\theta+ \frac{\mu-\theta-\sqrt{(\mu-\theta)^2+8(\mu\theta-1+\delta)}}{2}\ \ \frac{\mu-\theta+\sqrt{(\mu-\theta)^2+8(\mu\theta-1+\delta)}}{4}=1-\delta.
\end{align*}

Given $(x,x^*)\in \Pi(X)$, write $x=(y,z)\in Y\oplus_1 Z$, $x^*=(y^*,z^*)\in Y\oplus_\infty Z$ and observe that
\begin{equation}\label{prop: ell1.1}
 \norm{y}+\norm{z}=1=\re x^*(x)= \re y^*(y)+\re z^*(z)\leq \norm{y^*}\norm{y}+\norm{z^*}\norm{z}.
\end{equation}
If it holds $\left\|(1-\Psi(\mu,\theta,\delta))y_0^*-y^*\right\|\geq \Psi(\mu,\theta,\delta)$
then $\|x_0^*-x^*\|\geq \Psi(\mu,\theta,\delta)$ and we are done. If otherwise we have that
$\left\|(1-\Psi(\mu,\theta,\delta))y_0^*-y^*\right\|< \Psi(\mu,\theta,\delta)$ then we can write
$$
\big||1-\Psi(\mu,\theta,\delta)|-\|y^*\|\big|\leq\left\|(1-\Psi(\mu,\theta,\delta))y_0^*-y^*\right\|< \Psi(\mu,\theta,\delta).
$$
Now the hypothesis $\mu\theta\leq 2(1-\delta)$ gives us that $1-\Psi(\mu,\theta,\delta)\geq0$ and hence
$$
\abs{1-\Psi(\mu,\theta,\delta)-\norm{y^*}}<\Psi(\mu,\theta,\delta).
$$
From this it follows that $\norm{y^*}<1$ and so, $y=0$ by (\ref{prop: ell1.1}), giving $\norm{z}=1$. But then
$$
\norm{x_0-x}=k\mu\|y_0\|+\|\mu(1-k)z_0-z\|\geq \Abs{\mu(1-k)- \norm{z}} + k\mu= (2k-1)\mu+1=\Psi(\mu,\theta,\delta)
$$
which finishes the proof.
\end{proof}

The above proposition can be applied to vector-valued $L_1$ spaces, providing the following family of examples.

\begin{example}
Let $(\Omega, \Sigma,\nu)$ be a measure space containing two disjoint measurable sets with positive and finite measure and let $X$ be a Banach space.  Then, $\Phi_{L_1(\nu,X)}(\mu,\theta,\delta)=\Psi(\mu,\theta,\delta)$ for $\delta\in (0,1)$ and $\mu,\theta\in [0,1]$ with $1-\delta< \mu\theta \leq 2(1-\delta)$.
\end{example}

As it may be expected, a dual argument to the one given in Proposition~\ref{prop:ell1-sum} allows us to deduce an analogous result for a Banach space which decomposes as an $\ell_\infty$-sum. In fact we get a better result using ideals instead of subspaces. Given a Banach space $X$ we will write $w^*$ to denote the weak$^*$-topology $\sigma(X^*,X)$ in $X^*$.

\begin{prop}\label{prop:ideals}
Let $X$ be a Banach space. Suppose that $X^*=Y\oplus_1 Z$ where $Y$ and $Z$ are (non-trivial) subspaces of $X^*$ such that $\overline{Y}^{w^*}\!\neq X^*$ and $\overline{Z}^{w^*}\!\neq X^*$ .
Let $\delta\in (0,1)$ and $\mu,\theta\in [0,1]$ with $1-\delta< \mu\theta \leq 2(1-\delta)$. Then, there exists a pair $(x_0,x_0^*)\in X\times X^*$ with $\norm{x_0}=\mu$, $\norm{x_0^*}=\theta$ and $\re x_0^*(x_0)\geq 1-\delta$ satisfying
$$
d_\infty\big((x_0,x_0^*),\Pi(X)\big)=\Psi(\mu,\theta,\delta).
$$
Therefore, $\Phi_X(\mu,\theta,\delta)=\Psi(\mu,\theta,\delta)$ for the cited values of $\delta$, $\mu$, $\theta$.
\end{prop}

\begin{proof} Since $0<1-\delta<\mu\theta$ we get that $\theta>0$, so we can take $k = \frac{\theta-\mu+\sqrt{(\mu -\theta)^2+8( \mu\theta-1+\delta)}}{4\theta}$ which satisfies $0\leq k\leq 1$ because $\delta<1$.

As it is observed in the proof of \cite[Proposition~4.6]{C-K-M-M-R} it is possible to find $y_0,z_0\in S_X$ and $y_0^*\in S_{Y}$ and $z_0^*\in S_Z$ such that
\begin{equation*}
\re y_0^*(y_0)=1,\quad \re z_0^*(z_0)=1,\quad y^*(z_0)=0 \quad \forall
y^*\in Y,\quad z^*(y_0)=0 \quad \forall z^*\in Z.
\end{equation*}
We now define
\begin{equation*}
x^*_0=(k\theta y_0^*, (1-k)\theta z_0^*) \in X^*
\qquad x_0=(1-\Psi(\mu,\theta,\delta))y_0+\mu z_0\in X
\end{equation*}
and first we observe that
\begin{align*}
\re x_0^*(x_0)&=\mu\theta+ (1-\mu-\Psi(\mu,\theta,\delta))\theta k\\
&=\mu\theta+ \frac{\theta-\mu-\sqrt{(\mu-\theta)^2+8(\mu\theta-1+\delta)}}{2}\ \ \frac{\theta-\mu+\sqrt{(\mu-\theta)^2+8(\mu\theta-1+\delta)}}{4}=1-\delta.
\end{align*}
Besides, since $0\leq k\leq 1$, it is clear that $\norm{x_0^*}=\theta$. Let us check that $\|x_0\|=\mu$. Indeed, using the fact that  $|1-\Psi(\mu,\theta,\delta)|\leq \mu$, for every
$x^*=y^*+z^*\in S_{X^*}$ one has
\begin{align*}
|x^*(x_0)|=\left|(y^*+z^*)((1-\Psi(\mu,\theta,\delta)) y_0+\mu z_0)\right|\leq |1-\Psi(\mu,\theta,\delta)|\norm{y^*}+\mu\norm{z^*}\leq \mu(\norm{y^*}+\norm{z^*})=\mu.
\end{align*}
This, together with $|z_0^*(x_0)|=\mu$,  gives $\|x_0\|=\mu$.

Let $(x,x^*)\in \Pi(X)$. We consider the semi-norm $\|\cdot\|_Y$ defined on $X$ by
$\|x\|_Y:=\sup\{\abs{y^*(x)} \, : \, y^*\in S_{Y^*}\}$ which is
smaller than or equal to the original norm, write $x^*=y^*+z^*$ with
$y^*\in Y$ and $z^*\in Z$, and observe that
\begin{equation}\label{eq:dual-ell1-sum}
\|y^*\|+\|z^*\|=1=\re x^*(x)=\re y^*(x) + \re z^*(x) \leq \|y^*\|\|x\|_Y +
\|z^*\|\|x\|.
\end{equation}
If $\|x_0-x\|_Y \geq \Psi(\mu,\theta,\delta)$ we obviously have $\dist((x_0,x_0^*),\Pi(X))\geq\Psi(\mu,\theta,\delta)$.

Otherwise $\|x_0-x\|_Y < \Psi(\mu,\theta,\delta)$, and we can write
\begin{align*}
\big||1-\Psi(\mu,\theta,\delta)|-\norm{x}_Y\big|\leq \norm{(1-\Psi(\mu,\theta,\delta))y_0-x}_Y =\norm{x_0-x}_Y<\Psi(\mu,\theta,\delta).
\end{align*}
Now, the hypothesis $\mu\theta\leq 2(1-\delta)$ gives us that $1-\Psi(\mu,\theta,\delta)\geq0$ and hence
$$
\abs{1-\Psi(\mu,\theta,\delta)-\norm{x}_Y}<\Psi(\mu,\theta,\delta).
$$
From this it follows that $\norm{x}_Y<1$ and so, $y^*=0$ by \eqref{eq:dual-ell1-sum}, giving $\norm{z^*}=1$. But then
$$
\norm{x_0^*-x^*}=k\theta\|y_0^*\|+\|\theta(1-k)z_0^*-z^*\|\geq k\theta+ \Abs{\theta(1-k)-\norm{z^*}}= (2k-1)\theta+1=\Psi(\mu,\theta,\delta)
$$
which finishes the proof.
\end{proof}

As an immediate consequence, we obtain the mentioned result for Banach spaces which decompose as an $\ell_\infty$ sum of two non-trivial subspaces.

\begin{corollary}\label{sharp-mod 2}
Let $X$ be a Banach space. Suppose that there are two (non-trivial) subspaces $Y$ and $Z$ such that $X=Y\oplus_\infty Z$. Let $\delta\in (0,1)$ and $\mu,\theta\in [0,1]$ with $1-\delta< \mu\theta \leq 2(1-\delta)$. Then, there exists a pair $(x_0,x_0^*)\in X\times X^*$ with $\norm{x_0}=\mu$, $\norm{x_0^*}=\theta$ and $\re x_0^*(x_0)\geq 1-\delta$ satisfying
$$
d_\infty\big((x_0,x_0^*),\Pi(X)\big)=\Psi(\mu,\theta,\delta).
$$
Therefore, $\Phi_X(\mu,\theta,\delta)=\Psi(\mu,\theta,\delta)$ for the cited values of $\delta$, $\mu$, $\theta$.
\end{corollary}

This corollary applies to vector-valued $L_\infty$ spaces and vector-valued $c_0$ spaces.

\begin{examples}
\begin{enumerate}
\item[(a)] Let $(\Omega, \Sigma,\nu)$ be a measure space containing two disjoint measurable sets with positive measure and let $X$ be a Banach space.  Then, $\Phi_{L_\infty(\nu,X)}(\mu,\theta,\delta)=\Psi(\mu,\theta,\delta)$ for $\delta\in (0,1)$ and $\mu,\theta\in [0,1]$ with $1-\delta< \mu\theta \leq 2(1-\delta)$.
\item[(b)] Let $\Gamma$ be a set with at least two points and let $X$ be a non-trivial Banach space.  Then, $\Phi_{c_0(\Gamma,X)}(\mu,\theta,\delta)=\Psi(\mu,\theta,\delta)$ for $\delta\in (0,1)$ and $\mu,\theta\in [0,1]$ with $1-\delta< \mu\theta \leq 2(1-\delta)$.
\end{enumerate}
\end{examples}

Moreover, Proposition~\ref{prop:ideals} allows to get the result for vector-valued $C_0(L)$ spaces using the concept of $M$-ideal. Using the same ideas provided in \cite[Corollary~4.9 and Example~4.10]{C-K-M-M-R}, we get the following family of examples.

\begin{example}
Let $L$ be a locally compact Hausdorff topological space with at least two points and let $X$ be a Banach space.  Then, $\Phi_{C_0(L,X)}(\mu,\theta,\delta)=\Psi(\mu,\theta,\delta)$ for $\delta\in (0,1)$ and $\mu,\theta\in [0,1]$ with $1-\delta< \mu\theta \leq 2(1-\delta)$.
\end{example}

\subsection{Hilbert spaces}

We deal first with the simplest example, $X=\R$.

\begin{prop}
Let $\delta \in (0,2)$, $x,x^*\in \R$ such that $|x|,|x^*|\leq 1$ with $x^*x>1-\delta$, then
$$
d_\infty\big((x,x^*),\Pi(\R)\big)\leq \begin{cases} 1-\min\{|x|,|x^*|\} & \text{ if } 0<\delta\leq1 \\
1+\min\{|x|,|x^*|\} & \text{ if } 1\leq\delta<2.
\end{cases}
$$
Moreover, this inequality is sharp.  Given $\mu, \theta \in [0,1)$ with $\mu\theta\geq 1-\delta$ there exists a pair $(x,x^*)\in \R \times \R$ with $|x|\leq\mu$, $|x^*|\leq\theta$ and $x^*x\geq1-\delta$ satisfying
\[
d_\infty\big((x,x^*),\Pi(\R)\big)=\begin{cases} 1-\min\{|x|,|x^*|\} & \text{ if } 0<\delta\leq1 \\
1+\min\{|x|,|x^*|\} & \text{ if } 1\leq\delta<2.
\end{cases}
\]
\end{prop}

\begin{proof}
Fix $x,x^*\in [-1,1]$ with $x^* x> 1-\delta$. We take $y=y^*\in \{-1,1\}$ to be the sign of the number in $\{x,x^*\}$ which has bigger modulus (in case $|x|=|x^*|$ any choice will do).

Suppose first that $\delta\in (0,1)$. In this case $x$ and $x^*$ have the same sign. Hence, we have that
$|x-y|=1-|x|$ and $|x^*-y^*|=1-|x^*|$. So $d_\infty\big((x,x^*),\Pi(\R)\big)\leq 1-\min\{|x|,|x^*|\}$.

Suppose now that $\delta\in[1,2)$. The choice of $y$ and $y^*$ allows us to write
\[
d_\infty\big((x,x^*),\Pi(\R)\big)\leq\max\{|x-y|,|x^*-y^*|\}\leq 1+\min\{|x|,|x^*|\}.
\]

To prove the moreover part, suppose first that $\delta\in(0,1]$ and observe that $x=\mu$, $x^*=\theta$ satisfy the desired conditions. When $\delta\in (1,1+\mu\theta)$ we have that $\mu\theta>0$. So we can define $x=\mu$ and $x^*=\frac{1-\delta}{\mu}$ which fulfill the requirements.
Finally, when $\delta\in [1+\mu\theta,2)$ the elements $x=\mu$ and $x^*=-\theta$ do the job.
\end{proof}

Our goal now is to deal with (real) Hilbert spaces of dimension greater than one.

Let $H$ be a real Hilbert space. Taking into account that $H^*$ can be identified with $H$, and that the action of a vector $y\in H$ on a vector $x\in H$ is given by their inner product $\langle x , y\rangle$, we can write
$$
\Pi(H)=\{(z,z)\in S_H\times S_H\}.
$$
In the next result, fixed a pair $(x,y)\in B_H\times B_H$, we obtain the distance of $(x,y)$ to $\Pi(H)$ in terms of $\|x\|$, $\|y\|$,  and $\langle x, y\rangle $.

\begin{theorem}\label{thm:hilbert2}
Let $H$ be a real Hilbert space with $\dim(H)\geq2$ and let $x,y$ be different points in $B_H$ with $\|x\|\geq \|y\|$. Then,
$$
d_\infty\big((x,y),\Pi(H)\big)=\begin{cases}
1-\|y\| & \text{ if } \langle x , y \rangle \geq \|y\|^2+\|y\| \frac{\|x\|^2-\|y\|^2}{2}, \\
\sqrt{1-\langle x, y \rangle -2\lambda\sqrt{\|x\|^2\|y\|^2-\langle x,y \rangle^2} } & \text{ if } \langle x , y \rangle < \|y\|^2+\|y\| \frac{\|x\|^2-\|y\|^2}{2},
\end{cases}
$$
where
$$
\lambda=\frac{-2\sqrt{\|x\|^2\|y\|^2-\langle x, y \rangle^2}+\sqrt{4\big(\|x\|^2+\|y\|^2-2\langle x,y \rangle\big)-(\|x\|^2-\|y\|^2)^2}}{2\big(\|x\|^2+\|y\|^2-2\langle x, y \rangle\big) }
$$
\end{theorem}

We will need the following easy observations.

\begin{lemma}\label{lemma:Hilbert-dim2}
Let $\alpha_0\in]-\pi,\pi]$, $a\geq0$, $b\geq0$, and let $f:[\alpha_0-\pi,\alpha_0+\pi]\longrightarrow \R$ be defined by
$$
f(\alpha)=\|(a\cos(\alpha_0),a\sin(\alpha_0))-(b\cos(\alpha),b\sin(\alpha))\|_2.
$$
If $ab>0$ then $f$ decreases in $[\alpha_0-\pi,\alpha_0]$ and increases in $[\alpha_0,\alpha_0+\pi]$. If $ab=0$ then $f$ is constant.
\end{lemma}

\begin{proof}Only the case $ab>0$ needs an explanation.
Taking into account that $f^2(\alpha)=a^2+b^2-2ab\cos(\alpha_0-\alpha)$, it suffices to observe that
$ab\cos(\alpha_0-\alpha)$ increases in $[\alpha_0-\pi,\alpha_0]$ and decreases in $[\alpha_0,\alpha_0+\pi]$.
\end{proof}

\begin{remark}\label{remark:minimum-distance-point-cfa}
Lemma~\ref{lemma:Hilbert-dim2} is telling us in particular that, given a circle $C$ and a point $x$ in the same plane which is not the center of $C$, the minimum distance from $x$ to $C$ is attained at the intersection point of $C$ and the half-line starting at the center of $C$ which passes through $x$.
\end{remark}

\begin{proof}[Proof of Theorem~\ref{thm:hilbert2}] If $y=0$ we have to show that $d_\infty\big((x,y),\Pi(H)\big)=1$, but this clear since, obviously,  $d_\infty((x,0),(\frac{x}{\|x\|},\frac{x}{\|x\|}))\leq 1$ and every $h\in S_H$ satisfies $d_\infty((x,0),(h,h))\geq \|h\|= 1$. So we can set $y\neq0$ for the rest of the proof.

In the next step we show that we can reduce the problem to the 2-dimensional case. Let $X$ be the 2-dimensional subspace of $H$ containing $x$ and $y$. We claim that $d_\infty\big((x,y),\Pi(H)\big)=d_\infty\big((x,y),\Pi(X)\big)$. Indeed, since $\Pi(X)\subset \Pi(H)$ the inequality $d_\infty\big((x,y),\Pi(H)\big)\leq d_\infty\big((x,y),\Pi(X)\big)$ is evident. To prove the reversed inequality, fixed $h\in S_H$, consider the plane $P$ which contains $h$, intersects $X$ in a line and which is orthogonal to the line containing $x$ and $y$. Set $h_X\in X$ to be the intersection point of $P$ and the line containing $x$ and $y$. We observe that $P\cap S_H$ is a circle which contains $h$ and we write $\widetilde h_X$ to denote the intersection point of $P\cap S_H$ and the half-line starting at the centre of $P\cap S_H$ and containing $h_X$. If $h_X$ happens to be the centre of $P\cap S_H$, any of the two points in $P\cap S_H \cap X$ can be taken as $\widetilde h_X$. By Remark~\ref{remark:minimum-distance-point-cfa} we have that $\|h-h_X\|\geq \|\widetilde h_X- h_X\|$. Finally, using the orthogonality between $P$ and the line containing $x$ and $y$, we can write
$$
\|x-h\|=(\|x-h_X\|^2+\|h_X-h\|^2)^{1/2}\geq(\|x-h_X\|^2+\|h_X-\widetilde h_X\|^2)^{1/2}=\|x-\widetilde h_X\|
$$
and, similarly $\|y-h\|\geq \|y-\widetilde h_X\|$. Therefore, we get $d_\infty((x,y),(h,h))\geq d_\infty\big((x,y),\Pi(X)\big)$
and taking infimum for $h\in S_H$ we obtain the desired inequality. Thus, we can suppose that $H$ is 2-dimensional and we can identify it with $(\R^2,\|\cdot\|_2)$.

Set $\widetilde x=\frac{x}{\|x\|}$ and $\widetilde y=\frac{y}{\|y\|}$. Of the two points in $S_H$ whose distances to $x$ and $y$ are equal, let $m$ be the one that minimizes that distance. We claim that $d_\infty\big((x,y),\Pi(X)\big)$ is attained at one of the three pairs $(\widetilde x,\widetilde x),(\widetilde y,\widetilde y)$ or $(m,m)$. Indeed, for $h\in S_H$ denote $f_x(h) = \|x - h\|$, $f_y(h) = \|y - h\|$, and $f(h) = \max\{f_x(h), f_y(h)\}$. It is clear that $f$ attains its minimum, say that it does at $h_0\in S_H$. Then $h_0$ must be either a point of local minimum of $f_x$, or  a point of local minimum of $f_y$, or it satisfies $f_x(h_0) = f_y(h_0)$. Lemma~\ref{lemma:Hilbert-dim2} tells us that the only local minimum for $f_x$ is $\widetilde{x}$ and the only local minimum for $f_y$ is $\widetilde{y}$. So $h_0$ must one of the following four points: $\widetilde x$, $\widetilde y$, $m$ and the remaining point $p$ of $S_H$ whose distances to $x$ and $y$ are equal, but for sure $f(p)$ is not the minimal value, so we omit this possibility.

To obtain the value of $d_\infty\big((x,y),\Pi(X)\big)$ we have to determine which is the suitable pair among $(\widetilde x,\widetilde x),(\widetilde y,\widetilde y)$, and $(m,m)$. We distinguish two cases depending on the value of $\langle x, y \rangle $:

If $\langle x , y \rangle \geq \|y\|^2+\|y\| \frac{\|x\|^2-\|y\|^2}{2}$ then
$$
\|y-\widetilde{y}\|^2=(1-\|y\|)^2\geq \|x\|^2+1-\frac{2}{\|y\|}\langle x, y \rangle=\|x-\widetilde{y}\|^2
$$
which gives us that $\|y-\widetilde{y}\|\geq \|x-\widetilde{y}\|$, and so $d_\infty\big((x,y),(\widetilde{y},\widetilde{y})\big)=\|y-\widetilde{y}\|=1-\|y\|$. On the other hand, Remark~\ref{remark:minimum-distance-point-cfa} tells us that $\|y-\widetilde{y}\|\leq \text{dist}(y,S_H)$. Therefore, we can write
$$
d_\infty\big((x,y),(\widetilde{y},\widetilde{y})\big)=\|y-\widetilde{y}\|\leq \text{dist}(y,S_H)\leq d_\infty((x,y), \Pi(H))\leq d_\infty\big((x,y),(\widetilde{y},\widetilde{y})\big),
$$
finishing the proof in this case.

Suppose otherwise that $\langle x , y \rangle < \|y\|^2+\|y\| \frac{\|x\|^2-\|y\|^2}{2}$\,. Then we obtain
$\|y-\widetilde{y}\|<\|x-\widetilde{y}\|$, and thus $d_\infty\big((x,y),(\widetilde{y},\widetilde{y})\big)=\|x-\widetilde{y}\|$. We observe that since $\|y\|\leq\|x\|$ and $\|x\|+\|y\|\leq 2$ we also have
$$
\langle x , y \rangle < \|y\|^2+\|y\| \frac{\|x\|^2-\|y\|^2}{2}\leq\|x\|^2+\|x\| \frac{\|y\|^2-\|x\|^2}{2}\,.
$$
Hence, we can deduce analogously that $\|x-\widetilde{x}\|<\|y-\widetilde{x}\|$ and so
$d_\infty\big((x,y),(\widetilde{x},\widetilde{x})\big)=\|y-\widetilde{x}\|$.

Let us check that in this case one has $d_\infty\big((x,y),(m, m)\big)\leq\min\big\{\|x-\widetilde{y}\|,\|y-\widetilde{x}\|\big\}$ and, therefore, $(m,m)$ is the suitable pair.
We start observing that, up to a rotation, we can assume without loss of generality that
$$
x=(a_x\cos(\alpha_x), a_x\sin(\alpha_x))\qquad  \text{and} \qquad  y=(a_y\cos(\alpha_y), a_y\sin(\alpha_y)),
$$
where $a_x> 0$, $a_y> 0$, $\alpha_x, \alpha_y\in [0,\pi]$, and $\alpha_x \leq \alpha_y$.
Then, by Lemma~\ref{lemma:Hilbert-dim2}, the function $f_x:[\alpha_x,\alpha_y]\longrightarrow \R$ given by $f_x(\alpha)=\|(a_x\cos(\alpha_x),a_x\sin(\alpha_x))-(\cos(\alpha), \sin(\alpha))\|$ is increasing and the function $f_y:[\alpha_x,\alpha_y]\longrightarrow \R$ given by $f_y(\alpha)=\|(a_y\cos(\alpha_y),a_y\sin(\alpha_y))-(\cos(\alpha), \sin(\alpha))\|$ is decreasing. Besides, we have that
$$
f_x(\alpha_x)=\|x-\widetilde x\|<\|y-\widetilde x\|=f_y(\alpha_x) \qquad \text{and} \qquad f_y(\alpha_y)=\|y-\widetilde y\|<\|x-\widetilde y\|=f_x(\alpha_y).
$$
So there is $\alpha_1\in (\alpha_x,\alpha_y)$ satisfying $f_x(\alpha_1)=f_y(\alpha_1)$. Obviously one has that $m=(\cos(\alpha_1),\sin(\alpha_1))$,
$$
\|x-m\|=f_x(\alpha_1)=f_y(\alpha_1)<f_y(\alpha_x)= \|y-\widetilde x\|,\qquad \text{and} \qquad \|y-m\|=f_x(\alpha_1)<f_x(\alpha_y)=\|x-\widetilde y\|.
$$
We finish the proof computing $d_\infty\big((x,y),(m,m)\big)$. To this end, we write $x=(x_1,x_2)$, $y=(y_1,y_2)$, and $z=(y_2-y_1,x_2-x_1)$ which is orthogonal to $x-y$ and obviously satisfies $\|z\|=\|x-y\|$. We can assume without loss of generality (exchanging $z$ by $-z$ if necessary) that $\langle x-y, z \rangle \geq 0$. With this notation we can write $m=\frac{x+y}{2}+\lambda z$ for suitable $\lambda$ that we have to compute. Since $m$ must be in $S_H$ we obtain the following equation for $\lambda:$
$$
1=\|m\|^2=\frac{\|x+y\|^2}{4}+\lambda\langle x+y,z \rangle +\lambda^2\|x-y\|^2.
$$
Besides, observe that
$$
\langle x+y,z \rangle^2=4(x_1y_2-x_2y_1)^2=4(\|x\|^2\|y\|^2-\langle x, y \rangle^2)
$$
and, therefore,
\begin{equation}\label{eq:lambda}
1=\frac{\|x+y\|^2}{4}+2\lambda\sqrt{\|x\|^2\|y\|^2-\langle x, y \rangle^2} +\lambda^2\|x-y\|^2.
\end{equation}
Observe further that
$$
\|y-m\|^2=\|x-m\|^2=\left\|\frac{x-y}{2}+\lambda z\right\|^2=\frac{\|x-y\|^2}{4}+\lambda^2\|x-y\|^2.
$$
Hence, we have to pick $\lambda$ to be the solution of \eqref{eq:lambda} which has smaller modulus, that is:
$$
\lambda=\frac{-2\sqrt{\|x\|^2\|y\|^2-\langle x, y \rangle^2}+\sqrt{4\big(\|x\|^2\|y\|^2-\langle x, y \rangle^2\big)+4\|x-y\|^2-\|x+y\|^2\|x-y\|^2}}{2\|x-y\|^2}\,.
$$
Taking into account that
\begin{align*}
\|x+y\|^2\|x-y\|^2&=\big(\|x\|^2+\|y\|^2+2\langle x, y \rangle\big)\big(\|x\|^2+\|y\|^2-2\langle x, y \rangle\big)=\big(\|x\|^2+\|y\|^2\big)^2-4\langle x, y \rangle^2
\end{align*}
we get
$$
4\big(\|x\|^2\|y\|^2-\langle x, y \rangle^2\big)-\|x+y\|^2\|x-y\|^2=-\big(\|x\|^2-\|y\|^2\big)^2
$$
This, together with $\|x-y\|^2=\|x\|^2+\|y\|^2-2\langle x, y\rangle$, gives the expected value for $\lambda$. Finally, using \eqref{eq:lambda} we obtain
\begin{align*}
d_\infty\big((x,y),(m,m)\big)=\|x-m\|&=\sqrt{\frac{\|x-y\|^2}{4}+\lambda^2\|x-y\|^2}  \\
&=\sqrt{\frac{\|x-y\|^2}{4}+1-\frac{\|x+y\|^2}{4}-2\lambda\sqrt{\|x\|^2\|y\|^2-\langle x, y \rangle^2}}\\
&=\sqrt{1-\langle x, y \rangle-2\lambda\sqrt{\|x\|^2\|y\|^2-\langle x, y \rangle^2}}
\end{align*}
which finishes the proof.
\end{proof}

We may rewrite Theorem~\ref{thm:hilbert2} to provide the following computation of $\Phi_H(\mu,\theta,\delta)$.

\begin{corollary}
Let $H$ be a real Hilbert space with $\dim(H)\geq 2$, $\delta\in (0,2)$, and $\mu,\theta \in [0,1]$ satisfying  $\mu\geq \theta$ and $\mu\theta>1-\delta$. Then,
$$
\Phi_H(\mu,\theta,\delta)=\begin{cases}
1-\theta & \text{ if } 1-\delta \geq \theta^2+\theta \frac{\mu^2-\theta^2}{2}, \\
\max\left\{1-\theta,\sqrt{\delta-2\lambda_\delta\sqrt{\mu^2\theta^2-(1-\delta)^2}}\right\} & \text{ if } 1-\delta < \theta^2+\theta \frac{\mu^2-\theta^2}{2},
\end{cases}
$$
where
$$
\lambda_\delta=\frac{-2\sqrt{\mu^2\theta^2-(1-\delta)^2}+\sqrt{4\big(\mu^2+\theta^2-2+2\delta\big)-(\mu^2-\theta^2)^2}}{2\big(\mu^2+\theta^2-2+2\delta\big) }\,.
$$
\end{corollary}

\begin{proof}
Suppose first that $1-\delta \geq \theta^2+\theta \frac{\mu^2-\theta^2}{2}$ and fix an arbitrary pair $(x,y)\in H\times H$ with $\|x\|=\mu$, $\|y\|=\theta$ and $\langle x, y\rangle \geq 1-\delta$. Then, $\langle x, y\rangle\geq\theta^2+\theta \frac{\mu^2-\theta^2}{2}$ and Theorem~\ref{thm:hilbert2} gives $d_\infty\big((x,y),\Pi(H)\big)=1-\theta$, taking supremum in $(x,y)$ we obtain $\Phi_H(\mu,\theta,\delta)\leq 1- \theta$. The reversed inequality always holds by Remark~\ref{rem:lower-bound-Phi}.

Suppose now that $1-\delta< \theta^2+\theta \frac{\mu^2-\theta^2}{2}$. As we observed at the beginning of the proof of Theorem~\ref{thm:hilbert2}, we can suppose that $\dim(H)=2$ and so we can identify $H=(\R^2,\|\cdot\|_2)$. Fix an arbitrary pair $(x,y)\in H\times H$ with $\|x\|=\mu$, $\|y\|=\theta$ and $\langle x, y\rangle \geq 1-\delta$. Renaming $x$ and $y$ if necessary and using a suitable rotation, we can suppose without loss of generality that $x=(\mu,0)$ and $y=\theta(\cos(\alpha),\sin(\alpha))$ with $\alpha\in [0,\pi]$. Let $\alpha_1\in [0,\pi]$ be so that the point $z=\theta(\cos(\alpha_1),\sin(\alpha_1))$ satisfies $\langle x, z\rangle =\mu\theta\cos(\alpha_1)=1-\delta$. Observe that, in fact, one has $\alpha\in[0,\alpha_1]$.

Next, we write $\eps= \max\left\{1-\theta,\sqrt{\delta-2\lambda_\delta\sqrt{\mu^2\theta^2-(1-\delta)^2}}\right\}$
and we use Theorem~\ref{thm:hilbert2} for $x$ and $z$ to obtain
$$
d_\infty\big((x,z),\Pi(H)\big)= \sqrt{\delta-2\lambda_\delta\sqrt{\mu^2\theta^2-(1-\delta)^2}}.
$$
Let $\alpha_2\in [0,\pi]$ be so that the point $m=(\cos(\alpha_2),\sin(\alpha_2))$ satisfies
$$
\|x-m\|=\|z-m\|=\sqrt{\delta-2\lambda\sqrt{\mu^2\theta^2-(1-\delta)^2}}\,.
$$
If $\alpha\in [0,\alpha_2]$ then we can use Lemma~\ref{lemma:Hilbert-dim2} with $\alpha_0=0$, $a=\mu$, and $b=1$ to obtain that $\widetilde{y}=(\cos(\alpha),\sin(\alpha))$ satisfies
$$
\|x-\widetilde{y}\|=\big\|(\mu,0)-(\cos(\alpha),\sin(\alpha))\big\|\leq\big\|(\mu,0)-(\cos(\alpha_2),\sin(\alpha_2))\big\|=\|x-m\|
$$
and, therefore
$$
d_\infty\big((x,y),\Pi(H)\big)\leq \max\big\{\|x-\widetilde{y}\|,\|y-\widetilde{y}\|\big\}\leq\max\big\{\|x-m\|,1-\theta\big\}=\eps.
$$
If $\alpha\in[\alpha_2,\alpha_1]$ (obviously this case does not occur when $\alpha_2>\alpha_1$), we use Lemma~\ref{lemma:Hilbert-dim2} with $\alpha_0=\alpha_2$, $a=1$, and $b=\theta$ to obtain that
\begin{align*}
\|m-y\|&=\big\|(\cos(\alpha_2),\sin(\alpha_2))-(\theta\cos(\alpha),\theta\sin(\alpha))\big\|\\
&\leq\big\|(\cos(\alpha_2),\sin(\alpha_2))-(\theta\cos(\alpha_1),\theta\sin(\alpha_1))\big\|=\|m-z\|.
\end{align*}
This allows us to write
$$
d_\infty\big((x,y),\Pi(H)\big)\leq \max\big\{\|x-m\|,\|y-m\|\big\}\leq\max\big\{\|x-m\|,\|z-m\|\big\}\leq \eps.
$$
So, for every  $(x,y)\in H\times H$ with $\|x\|=\mu$, $\|y\|=\theta$ and $\langle x, y\rangle \geq 1-\delta$ we have $d_\infty\big((x,y),\Pi(H)\big)\leq\eps$ and, therefore, $\Phi_H(\mu,\theta,\delta)\leq \eps$. To prove the reversed inequality, it suffices to recall that $\Phi_H(\mu,\theta,\delta)\geq 1-\theta$ always holds and that
$\Phi_H(\mu,\theta,\delta)\geq d_\infty\big((x,z),\Pi(H)\big)=\sqrt{\delta-2\lambda_\delta\sqrt{\mu^2\theta^2-(1-\delta)^2}}$.
\end{proof}

\section[Estimation of the spherical modulus for uniformly non-square spaces]{Estimation of the spherical Bishop-Phelps-Bollob\'{a}s modulus for uniformly non-square spaces}\label{sec:non-square}

In \cite[Theorem~5.9]{C-K-M-M-R} it is proved that for a uniformly non-square space $X$ and $\delta\in(0,\frac12)$ one has
\begin{equation*}
\Phi_X^S(\delta) < \sqrt{2\delta}.
\end{equation*}
The proof of this fact is involved and it is impossible to extract from it any better estimate for $\Phi_X^S(\delta)$. In this section we obtain a smaller upper bound for $\Phi_X^S(\delta)$ by means of a parameter that measures the uniformly non-squareness of the space $X$. We recall that uniformly non-square spaces were introduced by James \cite{S4} as those spaces whose two-dimensional subspaces are uniformly separated from $\ell_1^{(2)}$.
The main result of \cite{S4} -- the reflexivity of uniformly non-square spaces -- was the origin of the theory of superreflexive spaces. Basing on James results one can prove even more: if $E$ is an arbitrary two-dimensional space and $X$ has the property that two-dimensional subspaces of $X$ are uniformly separated from $E$, then $X$ is reflexive \cite{Kad1}.

Recall that a Banach space $X$ is \emph{uniformly non-square} if and only if there is $\alpha>0$ such that
$$
\fr{1}{2}(\|x+y\|+\|x-y\|)\leq 2-\alpha
$$
for all $x, y \in B_X$. The \emph{parameter of uniform non-squareness} of $X$, which we denote $\alpha(X)$, is the best possible value of $\alpha$ in the above inequality. In other words,
$$
\alpha(X) := 2-\underset{x,y\in B_X}{\sup}\lf\fr{1}{2}(\|x+y\|+\|x-y\|)\rt.
$$
With this notation $X$ is uniformly non-square if and only if  $\alpha(X)>0$.

In the next result we obtain an upper bound for the parameter of uniform non-squareness.

\begin{prop}
$\alpha(X)\leq 2-\sqrt{2}$ for every Banach space $X$.
\end{prop}

\begin{proof}
According to the Day-Nordlander theorem \cite[p.~60]{S2}, the following estimate of the modulus of convexity $\delta_X(\varepsilon) = 1-\sup\lf\fr{\|x+y\|}{2}:x,y\in S_X,\|x-y\|=\varepsilon\rt$ of  an arbitrary Banach space $X$ holds true: $\delta_X(\varepsilon) \leq 1-\sqrt{1-\varepsilon^2/4}$.
Consequently, we can write

\begin{align*}
\alpha(X)&=2-\underset{x,y\in B_X}{\sup}\lf\fr{1}{2}(\|x+y\|+\|x-y\|)\rt \\
&=2-\underset{\eps\in (0,2]}{\sup}\left(\sup\lf\fr{\|x+y\|}{2}: x,y\in B_X, \|x-y\|=\eps\rt+\eps/2\right)\\
&\leq 2-\underset{\varepsilon\in (0,2]}{\sup}\left(\sup\lf\fr{\|x+y\|}{2}:x,y\in {S}_X, \|x-y\|=\varepsilon\rt+\varepsilon/2\right)  \\
&= 2-\underset{\varepsilon\in (0,2]}{\sup}\{1-\delta_X(\varepsilon)+\varepsilon/2\}\leq 2-\underset{\varepsilon\in (0,2]}{\sup}\lf\varepsilon/2+\sqrt{1-\varepsilon^2/4}\rt \\
&= 2 - \sqrt{2}.   \hfill \qedhere
\end{align*}
\end{proof}

\begin{prop}\label{prop:parameter-UNS-self-dual}
The parameter of uniform non-squareness is self-dual, i.e. $ \alpha(X)= \alpha(X^*) $ for every Banach space $X$.
\end{prop}

\begin{proof}
For arbitrary $x,y\in B_X$  consider supporting functionals  $f,g$ at the points $x+y$ and $x-y$ respectively, i.e.  $f,g \in S_{X^*}$ satisfying  $f(x+y) = \|x +y\|$ and  $g(x-y) = \|x -y\|$. Then,
\begin{align*}
\|f + g\| + \|f - g\| &\ge (f+g)(x) + (f-g)(y) \\
 &= f(x+y) + g(x-y) =  \|x +y\| +  \|x -y\|.
\end{align*}
Hence, we get
\begin{align*}
  \underset{f,g\in B_{X^*}}{\sup}\lf \|f+g\|+\|f-g\| \rt \geq  \|x+y\|+\|x-y\|.
\end{align*}
Moving $x,y\in B_X$, we get $\alpha(X^*) \leq  \alpha(X)$. Substituting $X^*$ instead of $X$ we get $\alpha(X^{**}) \leq  \alpha(X^*)$. In the case of  $\alpha(X)>0$, the space is reflexive, and the above inequalities imply the desired equality  $\alpha(X)= \alpha(X^*)$. In the remaining case of $\alpha(X)=0$, we have $0=\alpha(X)\geq\alpha(X^*)\geq 0$, which finishes the proof.
\end{proof}

We are ready to present the promised result. The upper bound for $\Phi_X^S(\delta)$ that we give below does not pretend to be close to the sharp estimate that, unfortunately, we could not achieve.

\begin{theorem}\label{theorem:estimation-for-non-square}
Let $X$ be a Banach space with $\alpha(X^*)>\tilde\alpha > 0$.
Then,
$$
\Phi^S_X(\delta) \leq \sqrt{2\delta}\,\sqrt{1-\frac{1}{3}\tilde\alpha} \qquad \text{for } \delta \in \left(0, \frac{1}{2} - \frac{1}{6}\tilde\alpha \right)
$$
and
$$
\Phi^S_X(\delta) \leq {2\delta} \qquad \text{for } \delta \in \left( \frac{1}{2} - \fr{1}{6}\tilde\alpha,  \frac{1}{2} \right).
$$
Consequently, since $ \alpha(X)= \alpha(X^*) $ for every Banach space $X$, one has
$$
\Phi^S_X(\delta) \leq \sqrt{2\delta}\,\sqrt{1-\frac{1}{3}\alpha(X)} \qquad \text{for } \delta \in \left(0, \frac{1}{2} - \frac{1}{6}\alpha(X) \right)
$$
and
$$
\Phi^S_X(\delta) \leq {2\delta} \qquad \text{for } \delta \in \left( \frac{1}{2} - \fr{1}{6}\alpha(X),  \frac{1}{2} \right).
$$
\end{theorem}

Prior to provide the proof of the theorem, we recall that it obviously implies the commented result from \cite{C-K-M-M-R}.

\begin{corollary}[\mbox{\cite[Theorem~5.9]{C-K-M-M-R}}]
Let $X$ be a Banach space. Suppose that $\Phi_X^S(\delta)=\sqrt{2\delta}$ for some $\delta\in (0,1/2)$. Then $X^*$  is not uniformly non-square (i.e.\ $X^*$ (and $X$ as well) contains almost isometric copies of $\ell_\infty^{(2)}$).
\end{corollary}

The next result, which may be of independent interest, contains most of the difficulties in the proof of the Theorem~\ref{theorem:estimation-for-non-square}.

\begin{lemma} \label{main-lemma}
Let $X$ be a Banach space with $\alpha(X^*)> \tilde\alpha$. Then for every $\delta \in (0,2)$, every $(x,x^*)\in S_X\times S_{X^*}$ with $\re x^*(x)> 1-\delta$, and  every $ k \in (0, \frac12]$ there is a pair $(y,y^*)\in \Pi(X)$ such that
$$
\|x-y\|\leq \frac{\delta}{k} \qquad \text{and} \qquad \|x^*-y^*\|\leq 2k - \frac{2}{3}k\tilde\alpha\,.
$$
\end{lemma}

\begin{proof}
Fixed $(x,x^*)\in S_X\times S_{X^*}$ with $\re x^*(x)> 1-\delta$, we use Lemma~\ref{lemma:corolario2.2-phelps} for $C=B_X$ and $\eta=\delta$ to find
$y_0^*\in X^*$ and $y \in Y$ such that
$$
\|y\| = 1, \quad y_0^*(y) = \| y_0^*\|, \quad  \|x-y\|\leq\frac{\delta}{k}\quad \text{and} \quad \|x^*-{y_0}^* \| \leq k.
$$
Denoting $y^*=\frac{{y_0}^*}{\|{y_0}^*\|}$ one obviously has $(y, y^*)\in \Pi(X)$,
\begin{equation}\label{eq:Lemma-UNS-|1-y_0^*|small}
\|y^*-y_0^*\|=\bigl\vert1-\|{y_0}^*\| \bigr\vert\leq\|x^*-y_0^*\| \leq k, \qquad \text{and}
\end{equation}
\begin{align}\label{eq:Lemma-UNS-|1-y_0^*|big}
 \bigl\vert 1-\|{y_0}^*\| \bigr\vert &= \|y^*-{y_0}^*\| = \|y^*-x^*+x^*-{y_0}^*\| \geq \|x^*-y^*\|-k.
\end{align}
Besides, it is clear that
$$
\|x^*-y^*\|\leq \|x^*-y^*_0\|+\|y^*_0-y^*\|\leq 2k
$$
and
$$
\|x^*-{y_0}^*\|\geq \|x^*-y^*\|-\|y^*-{y_0}^*\|\geq\|x^*-y^*\|-k.
$$
On the other hand, since $\alpha(X^*)>\tilde\alpha$, we have
\begin{equation}\label{eq:thm-UNS-y^*-v^*}
\fr{1}{2}(\|y^*+v^*\|+\|y^*-v^*\|) \leq 2 - \tilde\alpha
\end{equation}
for every $v^* \in B_{X^*}$. In order to prove the lemma, we need to find a suitable $v^*\in S_{X^*}$ that allows us to estimate $\|x^*-y^*\|$.
We consider two cases separately.

\emph{Case 1}: $\|{y_0}^*\|>1$.\newline
Define in this case ${v_0}^*=\frac{1}{k}x^*-(1+\frac{1}{k})y^*$ which clearly satisfies $\|{v_0}^*\|\geq 1$. Using that $\|{y_0}^*\|>1$, \eqref{eq:Lemma-UNS-|1-y_0^*|small}, and \eqref{eq:Lemma-UNS-|1-y_0^*|big} we get
$$
\big\vert k+1-\|{y_0}^*\|\big\vert=\big|k-(\|y_0^*\|-1)\big|=k-\big|\|y_0^*\|-1\big|\leq 2k-\|x^*-y^*\|
$$
and, therefore,
$$
\left\|{v_0}^*-\frac{x^*-y_0^*}{k}\right\|=\left\|-\left(1+\fr{1}{k}\right)y^*+\fr{1}{k}{y_0}^*\right\| = \frac{1}{k}\big\vert k+1-\|{y_0}^*\|\big\vert\leq 2-\frac{\|x^*-y^*\|}{k}\,.
$$
Let us take $v^*:=\frac{{v_0}^*}{\|{v_0}^*\|} \in S_{X^*}$.
Since  $\|\frac{x^*-y_0^*}{k}\|\leq 1$, we have that $\|v_0^*\|\leq 3-\frac{\|x^*-y^*\|}{k}$ and so
$$
\|v^*-{v_0}^*\|=\big\vert1-\|{v_0}^*\|\big\vert = \|{v_0}^*\| - 1 \leq 2-\frac{\|x^*-y^*\|}{k}\,.
$$
Hence, we can estimate as follows:
$$
\|y^*+v^*\|\geq\|y^*+{v_0}^*\|-\|{v_0}^*-v^*\|\geq \frac{\|x^*-y^*\|}{k} - \left(2-\frac{\|x^*-y^*\|}{k}\right)=\frac2k\|x^*-y^*\|-2 \qquad \text{and}
$$
$$
\|y^*-v^*\|\geq\|y^*-{v_0}^*\|-\|{v_0}^*-v^*\|=\left\|(2+\frac1k)y^*-\frac1kx^*\right\|-\|{v_0}^*-v^*\|\geq 2 - \left(2-\frac{\|x^*-y^*\|}{k}\right)=\frac{\|x^*-y^*\|}{k}.
$$
This, together with \eqref{eq:thm-UNS-y^*-v^*}, tells us that
$$
\frac{3}{2k}\|x^*-y^*\|-1\leq \frac12\left(\|y^*+v^*\|+\|y^*-v^*\|\right)\leq2-\tilde \alpha
$$
which gives
$$
\|x^*-y^*\|\leq 2k-\frac23k\tilde\alpha,
$$
finishing the proof in this case.

\emph{Case 2}: $\|{y_0}^*\|\leq 1$. \newline This time let us define  ${v_0}^*=\frac{1}{k}x^*+(1-\frac{1}{k})y^*$ which satisfies
$$
\|{v_0}^*\|\geq\left\vert \frac{1}{k}\|x^*\|-\vert 1 - \frac{1}{k} \vert \|y^*\|\right\vert
= \left\vert \frac{1}{k} - \frac{1}{k} +1 \right\vert = 1.
$$
Using $\|{y_0}^*\|\leq1$, \eqref{eq:Lemma-UNS-|1-y_0^*|small}, and \eqref{eq:Lemma-UNS-|1-y_0^*|big} we can write
$$
\big\vert k-1+\|{y_0}^*\|\big\vert=\big|k-(1-\|y_0^*\|)\big|=k-\big|1-\|y_0^*\|\big|\leq 2k-\|x^*-y^*\|
$$
and, therefore,
$$
\left\|{v_0}^*-\frac{x^*-y_0^*}{k}\right\|=\left\|\fr{k-1}{k}y^*+\fr{1}{k}{y_0}^*\right\|=\frac1k\big\vert k-1+\|{y_0}^*\|\big\vert\leq 2-\frac{\|x^*-y^*\|}{k}\,.
$$
Let us take $ v^*:=\frac{{v_0}^*}{\|{v_0}^*\|} \in S_{X^*} $. Since $\|\frac{x^*-y_0^*}{k}\|\leq 1$, we have that $\|v_0^*\|\leq 3-\frac{\|x^*-y^*\|}{k}$ and so
$$
\|v^*-{v_0}^*\| = \big\vert1-\|{v_0}^*\|\big\vert \leq 2-\frac{\|x^*-y^*\|}{k}\,.
$$
On the one hand, we have that
$$
\|y^*-{v_0}^*\|=\frac{\|x^*-y^*\|}{k}
$$
and hence,
\begin{equation} \label{eq:thm-UNS-2-case}
\|y^*-v^*\|\geq\|y^*-{v_0}^*\|-\|{v_0}^*-v^*\|\geq \frac{2}{k}\|x^*-y^*\|-2 .
\end{equation}
On the other hand, using that $k\leq 1/2$, we can write
$$
 \|y^*+{v_0}^*\|=\frac1k\|x^*+(2k-1)y^*\| \geq  \frac{1}{k}\left(1 - |1 - 2k| \right)= 2
$$
and, therefore,
$$
\|y^*+v^*\| \geq\|y^*+{v_0}^*\|-\|{v_0}^*-v^*\|\geq \frac{\|x^*-y^*\|}{k}\,.
$$
This, together with \eqref{eq:thm-UNS-y^*-v^*} and \eqref{eq:thm-UNS-2-case}, allows us to write
$$
\frac{3}{2k}\|x^*-y^*\|-1\leq \frac12\left(\|y^*+v^*\|+\|y^*-v^*\|\right)\leq2-\tilde \alpha
$$
which again gives
$$
\|x^*-y^*\|\leq 2k-\frac23k\tilde\alpha
$$
and finishes the proof.
\end{proof}

\begin{proof}[Proof of Theorem~\ref{theorem:estimation-for-non-square}] Let $(x,x^*)\in S_X\times S_{X^*}$ with $\re x^*(x)> 1-\delta$ be fixed.
If $\delta \in \left(0, \frac{1}{2} - \fr{1}{6}\tilde\alpha \right)$ we take
$$
k = \sqrt{\frac{\delta}{2- \frac{2}{3}\tilde\alpha}}
$$
which satisfies $k < \frac{1}{2}$ and
$$
2k - \frac{2}{3}k\tilde\alpha =  \frac{\delta}{k} = \sqrt{2\delta}\,\sqrt{1-\frac{\tilde\alpha}{3}}.
$$
Hence, according to Lemma \ref{main-lemma}, we have that
$$
\|x-y\|\leq \sqrt{2\delta}\,\sqrt{1-\frac{\tilde\alpha}{3}}\qquad \text{and} \qquad
\|x^*-y^*\|\leq\sqrt{2\delta}\,\sqrt{1-\frac{\tilde\alpha}{3}}\,.
$$
Taking supremum in $(x,x^*)$ we get the desired inequality.
If otherwise $\delta \in \left( \frac{1}{2} - \fr{1}{6}\tilde\alpha,  \frac{1}{2} \right)$, we apply Lemma \ref{main-lemma} with $k = \frac12$ to obtain
$$
\|x-y\|\leq2\delta \qquad \text{and} \qquad \|x^*-y^*\|\leq 1-\frac13\tilde\alpha<2\delta
$$
which finishes the proof.
\end{proof}

%\newpage

\end{document}